\documentclass{amsart}
\usepackage{amsmath}
\usepackage{amssymb}
\usepackage{amsthm}
\usepackage{enumerate}
\theoremstyle{definition}
\newtheorem{definition}{Definition}[section]
\theoremstyle{plain}
\newtheorem{lemma}[definition]{Lemma}

\newtheorem{proposition}[definition]{Proposition}

\theoremstyle{remark}

\makeatletter
\@namedef{subjclassname@2020}{
	\textup{2020} Mathematics Subject Classification}
\makeatother

\newcommand{\myDef}{\operatorname{Def}}

\newcommand{\mydim}{\operatorname{Dim}}

\newcommand{\mySet}{\mathbf{DIM}}
\newcommand{\myNum}{\mathfrak n_{\dim}}
\newcommand{\mySelf}{\mathfrak D}
\newcommand{\myfineq}{\sim_{\text{fin}}}
\newcommand{\topdim}{\dim_{\text{top}}}
\begin{document}
	\title[Dimension functions in weakly o-minimal structures]{Dimension functions in weakly o-minimal structures admitting strong cell decomposition}
	\author[M. Fujita]{Masato Fujita}
	\address{Department of Liberal Arts,
		Japan Coast Guard Academy,
		5-1 Wakaba-cho, Kure, Hiroshima 737-8512, Japan}
	\email{fujita.masato.p34@kyoto-u.jp}
	\thanks{The author was supported by JSPS KAKENHI Grant Number JP25K07109.}
	
	\begin{abstract}
		We investigate the cardinality $\myNum(\mathcal M)$ of the sets of dimension functions on weakly o-minimal structures $\mathcal M$ admitting strong cell decomposition.
	\end{abstract}
	
	
	
	\maketitle

\section{Introduction}\label{sec:intro}

Let $\mathcal M=(M,\ldots)$ be a structure.
For $d<n$, let $\Pi_{d}^n:M^n \to M^{n-d}$ be the coordinate projection of $M^n$ forgetting the last $d$ coordinates.
We abbreviate $\Pi_1^n$ by $\Pi$ when $n$ is clear from the context.
For every $x \in \Pi_d^n(X)$, the fiber $\{y \in M^d\;|\;(x,y) \in X\}$ is denoted by $X_x^{\Pi^n_d}$.
We write $X_x^{\Pi}$ instead of $X_x^{\Pi_1^n}$.
The following is the definition of dimension function employed in this paper.
\begin{definition}[{\cite{vdD2}}]\label{def:dimension}
	Let $\mathcal M=(M,\ldots)$ be a structure.
	Let $\myDef_n(\mathcal M)$ be the collection of all subsets of $M^n$ definable in $\mathcal M$.
	Put $\myDef(\mathcal M):=\bigcup_{n>0}\myDef_n(\mathcal M)$.
	A function $\dim: \myDef(\mathcal M) \to \mathbb N \cup \{-\infty\}$, where $\mathbb N$ is the set of nonnegative integers, is called a \textit{dimension function on $\mathcal M$} if the following conditions are satisfied for every $n >0$:
	\begin{enumerate}
		\item[(1)] $\dim(\emptyset)=-\infty$, $\dim(\{a\})=0$ for each $a \in M$, $\dim(M^1)=1$.
		\item[$(2)_n$] Let $X$ and $Y$ be in $\myDef_n(\mathcal M)$.
		The equality
		\begin{align*}
			\dim(X \cup Y)=\max\{\dim(X),\dim(Y)\}
		\end{align*}
		holds.
		In particular, we have $\dim X \leq \dim Y$ if $X \subseteq Y$.
		\item[$(3)_n$] $\dim (X)=\dim(X^\sigma)$ for each $X \in \myDef_n(\mathcal M)$ and each permutation $\sigma$ of $\{1,\ldots, n\}$, where $$X^\sigma:=\{(x_{\sigma(1)},\ldots,x_{\sigma(n)}) \in M^n\;|\; (x_1,\ldots,x_n) \in X\}.$$
		\item[$(4)_n$] Suppose $n>1$. Let $X \in \myDef_n(\mathcal M)$. Put $$X(i):=\{x \in \Pi(X)\;|\; \dim X_x^{\Pi} = i\}$$ for $i=0,1$.
		Then, $X(i) \in \myDef_{n-1}(\mathcal M)$ and 
		\begin{equation*}
			\dim(X \cap \Pi^{-1}(X(i)))=\dim(X(i))+i 
		\end{equation*}
	\end{enumerate}
	
	For the technical reason, we introduce additional notations.
	$M^0$ represent a singleton and we put $\dim M^0=0$.
	For every $n >0$, a (coordinate) projection of $M^n$ onto $M^0$ means the trivial map. 
	It is obvious that $\dim X \leq n$ for every definable subset $X$ of $M^n$.
\end{definition}
The set of dimension functions on $\mathcal M$ is denoted by $\mySet(\mathcal M)$ and its cardinality is denoted by $\myNum(\mathcal M)$.

We computed $\myNum(\mathcal M)$ of several ordered structures $\mathcal M$ in \cite{Fuji_dim}.
In \cite[Section 3.3]{Fuji_dim}, we consider the case where $\mathcal M$ is o-minimal and obtain several nontrivial results.
These results are generalized to the case where the structure is a weakly o-minimal structure admitting strong cell decomposition.
Although the author believes that these generalizations are worth being public, they are less appealing than the o-minimal case, and in order to follow the proofs of the generalizations, readers should be familiar with several technical concepts used in the study of weakly o-minimal structures. 
Therefore, the author decided to write a separate note for these generalizations.
They are discussed in Section \ref{sec:gen}.

Let $\mathcal M$ be a weakly o-minimal structure $\mathcal M=(M,<,\ldots)$ admitting strong cell decomposition.
In \cite{Wencel}, Wencel constructed an o-minimal structure $\overline{\mathcal M}=(\overline{M},<,\ldots)$ satisfying the following:
\begin{itemize}
	\item $M$ is dense in $\overline{M}$,
	\item $\overline{X} \cap M^n$ is $\mathcal M$-definable for $\overline{X} \in \myDef_n(\overline{\mathcal M})$;
	\item For every $X \in \myDef_n(\mathcal M)$, there exists $\overline{X} \in \myDef_n(\overline{\mathcal M})$ such that $X=\overline{X} \cap M^n$. 
\end{itemize}
The structure $\overline{\mathcal M}$ is called the \textit{canonical o-minimal extension} of $\mathcal M$.
The relation between $\myNum(\mathcal M)$ and $\myNum(\overline{\mathcal M})$ is clarified in Section \ref{sec:ori}.
The inequality $\myNum(\mathcal M) \geq \myNum(\overline{\mathcal M})$ holds, but the equality $\myNum(\mathcal M) = \myNum(\overline{\mathcal M})$ does not hold in general.

Throughout, unless explicitly stated, $\mathcal M=(M,<,\ldots)$ is a weakly o-minimal structure admitting strong cell decomposition and $\overline{\mathcal M}=(\overline{M},<,\ldots)$ is the canonical o-minimal extension of $\mathcal M$. 
`definable' means `definable in the structure with parameters.'
An open interval is a nonempty set of the form $\{x \in M\;|\; a < x < b\}$ for some $a,b \in M \cup \{\pm \infty\}$.
It is denoted by $(a,b)$ in this paper.
We use the same notation for the pair of elements.
Readers will not be confused due to this abuse of notations.
The set $M$ is equipped with the order topology induced from the order $<$. 
The Cartesian product $M^n$ is equipped with the product topology of the order topology.

This paper uses the notions defined in Andrew's paper Wencel's papers \cite{Wencel, Wencel2}.
Some of them such as canonical o-minimal extension, strong cells, basic cells and CE cells have long definitions.
Quoting these definitions makes this paper unnecessarily long.
To avoid this, we do not give these definitions in this paper.
Readers are recommended to read this paper having those references at hand.

\section{Generalization of Proposition 3.19 in \cite{Fuji_dim} to weakly o-minimal structures admitting strong cell decomposition}\label{sec:gen}

In this section, we generalize the results in \cite[Section 3.3]{Fuji_dim} to the case where $\mathcal M$ is a weakly o-minimal structure admitting strong cell decomposition.

\begin{lemma}\label{lem:weak0}
	Suppose $\mathcal M=(M,<,\ldots)$ is an expansion of a linear order.
	Let $X \subseteq M$ be a definable convex open subset of $M$.
	Let $\dim$ be a dimension function on $\mathcal M$.
	If $\dim I=0$ for every bounded open interval $I$ contained in $X$, the equality $\dim X=0$ holds.
\end{lemma}
\begin{proof}
	Apply \cite[Lemma 3.2]{Fuji_dim}.
	The proof is a routine and we omit the proof.
\end{proof}

\begin{lemma}\label{lem:weak_omin_basic}
	The dimension function on a weakly o-minimal structure is completely determined by its restrictions to the set of all bounded open intervals.
\end{lemma}
\begin{proof}
	Let $\dim$ be a dimension function on $\mathcal M$.
	Suppose that the restriction of $\dim$ to the set of all bounded open intervals are given.
	By \cite[Corollary 3.3]{Fuji_dim}, we only have to show that, for every nonempty bounded subset $X$ of $M$, $\dim X=1$ if and only if $X$ contains an open interval $I$ with $\dim I=1$. 
	
	The `if' part is obvious from Definition \ref{def:dimension}(2).
	To show that `only if' part, suppose that $\dim I=0$ for every open interval $I$ contained in $X$.
	We only have to show $\dim X=0$.
	If $X$ is a finite set, we have $\dim X=0$ by Definition \ref{def:dimension}(1,2).
	Since $\mathcal M$ is weakly o-minimal, $X$ is a union of a finite set and finitely many open convex sets.
	By Definition \ref{def:dimension}, we may assume that $X$ is open convex set.
	
	Let $a \in X$.
	Consider $Y \langle a \rangle :=\{x \in X\;|\; x>a\}$.
	Observe that, for every $b \in Y\langle a \rangle$, $\{x \in Y\langle a \rangle \;|\; x<b\}$ is the open interval $(a,b)$ contained in $X$.
	We have  $\dim \{x \in Y\langle a \rangle \;|\; x<b\}=0$ by the case assumption.
	We get $\dim Y\langle a \rangle=0$ by \cite[Lemma 3.2]{Fuji_dim}.
	Apply \cite[Lemma 3.2]{Fuji_dim} to the family $\{Y\langle a \rangle\}_{a \in X}$, then we get $\dim X=0$.
\end{proof}

We recall how to construct a new dimension function.

\begin{definition}[{\cite[Definition 3.11]{Fuji_dim}}]\label{def:new_dim}
	Let $\mathcal M$ be a structure and $\dim:\myDef(\mathcal M) \to \mathbb N \cup \{-\infty\}$ be a function satisfying conditions (1), $(2)_1$ and $(4)_n$ for $n>1$ in Definition \ref{def:dimension}.
	Let $I$ be a definable subset of $M$ with $\dim I=1$.
	
	Put $I_0:=I$ and $I_1 := M \setminus I$. 
	For every positive integer $n$, $0<k \leq n$ and $\tau \in \{0,1\}^n$, the $k$-th coordinate of $\tau$ is denoted by $\tau(k)$ and put $I \langle \tau \rangle:=\prod_{i=1}^{n} I_{\tau(i)}$.
	Recall that $\Pi:M^n \to M^{n-1}$ be the projection forgetting the last coordinate.
	We define the dimension function $$\mydim [ \dim, I ]:\myDef(\mathcal M) \to \mathbb N \cup \{-\infty\}$$ as follows:
	
	Let $X$ be a definable subset of $M^n$.
	\begin{itemize}
		\item We put $\mydim [ \dim, I ](X)=-\infty$ if $X=\emptyset$;
		\item Let $\tau \in \{0,1\}^n$ and assume that $\emptyset \neq X \subseteq I\langle \tau \rangle$.
		\begin{itemize}
			\item Suppose $n=1$. If $\tau(1)=0$, $\mydim [ \dim, I ](X):=\dim X$.
			
			Put $\mydim [ \dim, I ](X):=0$ if $\tau(1)=1$. 
			\item Suppose $n>1$. If $\tau(n)=0$, put 
			\begin{align*}
				&S_1(X):=\{x \in \Pi(X)\;|\; \dim X_x^{\Pi} =1 \} \text{ and }\\
				&S_0(X):=\Pi(X) \setminus S_1(X).
			\end{align*}
			We define $\mydim [ \dim, I ](X):=\max\{\dim S_0(X), \dim S_1(X)+1\}$.
			
			If $\tau(n)=1$, put $\mydim [ \dim, I ](X)=\mydim [ \dim, I ](\Pi(X))$.
		\end{itemize}
		\item If $X$ is not empty, we put $$\mydim [ \dim, I ](X):=\max\{\mydim[\dim, I](X \cap I\langle \tau \rangle)\;|\; \tau \in \{0,1\}^n\}.$$
	\end{itemize}
	
	In this paper, we mainly consider the case in which $\dim=\topdim$, where $\topdim$ is the topological dimension defined in \cite[Example 1.2]{Fuji_dim}.
	If $\dim=\topdim$, we write $\mydim[I]$ instead of $\mydim [ \dim, I ]$.
\end{definition}


The following lemma is a generalization of \cite[Lemma 3.14]{Fuji_dim}.
In \cite[Lemma 3.14]{Fuji_dim}, $\mathcal M$ is assumed to be o-minimal.
In order to prove this generalization, we need more effort than the proof of \cite[Lemma 3.14]{Fuji_dim}.
We give a complete proof for this lemma.
\begin{lemma}\label{lem:omin_dimform_weak}
	Let $\mathcal M=(M,<,\ldots)$ be a weakly o-minimal structure admitting strong cell decomposition.
	Every dimension function on $\mathcal M$ is of the form $\mydim[I]$ for some definable subset $I$ of $M$.
\end{lemma}
\begin{proof}
	Let $\dim$ be a dimension function on $\mathcal M$.
	Put $$I:=\{a \in M\;|\; \forall b \ (a<b) \rightarrow \dim((a,b))=1\}.$$
	$I$ is definable by Definition \ref{def:dimension}(4).
	\medskip
	
	\textbf{Claim 1.} For every $a \in I$, there exists $b>a$ such that $(a,b) \subseteq I$.
	\begin{proof}[Proof of Claim 1]
		Let $a \in I$.
		Assume for contradiction that such $b$ does not exist.
		Because of weak o-minimality, $M \setminus I$ is a union of a finite set and finitely many open convex sets.
		Non-existence of $b$ implies that there exists $c>a$ such that $(a,c) \cap I=\emptyset$. 
		Since $a \in I$, we have $\dim ((a,c))=1$.
		On the other hand, we have $\dim ((d,c))=0$ for every $a<d<c$ by the definition of $I$.
		By \cite[Lemma 3.2]{Fuji_dim}, we have $\dim((a,c))=0$, which is absurd.
	\end{proof}
	
	We want to show that $\dim=\mydim[I]$.
	For that purpose, by Lemma \ref{lem:weak_omin_basic}, we have only to show that, for every bounded open interval $J$, $\dim(J)=1$ if and only if $I \cap J$ has a nonempty interior.
	The `if' part is easy from the definition of $I$.
	We show the `only if' part.
	Suppose $I \cap J$ has an empty interior.
	We want to show $\dim J=0$.
	
	We show that $I \cap J$ is an empty set.
	Assume for contradiction that $I \cap J$ is not empty and has an empty interior.
	Because of weak o-minimality, $I \cap J$ is a finite set.
	Choose $a \in I \cap J$.
	Then we can choose $b>a$ such that $b \in J$ and $(a,b) \cap I=\emptyset$, which contradicts Claim 1.
	
	Let $J=(a,b)$, where $a,b \in M$.
	Let us consider the map $f:J \to \overline{M}$ given by $$f(x):=\sup\{y \in M\;|\; x<y<b, \dim((x,y))=0\}.$$
	It is definable by Definition \ref{def:dimension}(4).
	Since $J \subseteq M \setminus I$, we have $x < f(x)$ for every $x \in J$.
	Put $X \langle x \rangle :=\{ y \in M\;|\; x<y<f(x)\}$ for every $x \in J$.
	We have $\dim X \langle x \rangle=0$ by \cite[Lemma 3.2]{Fuji_dim}.
	By the strong monotonicity \cite[Lemma 2.6(b)]{Wencel}, we can decompose $J$ into a finite set $F$ and finitely many open convex sets $K_1,\ldots, K_k$ such that the restriction $f|_{K_i}$ of $f$ to $K_i$ is continuous and monotone for $1 \leq i \leq k$.
	
	We show that $f|_{K_i}$ is constant.
	Let $x,x' \in K_i$ with $x<x'<f(x)$.
	We have $\dim ((x',y))=0$ for every $y \in J$ with $x'<y<f(x)$ by Definition \ref{def:dimension}(2) because $\dim((x,y))=0$ by the definition of $f(x)$.
	This observation implies $f(x) \leq f(x')$.
	Therefore, $f|_{K_i}$ is either constant or strictly increasing.
	Assume for contradiction that $f|_{K_i}$ is strictly increasing.
	Let $x,x' \in K_i$ with $x<x'<f(x)$.
	We have $f(x')>f(x)$.
	Choose $y \in J$ with $f(x)<y<f(x')$.
	We have $\dim((x,x'))=0$ because $x<x'<f(x)$ and $\dim((x',y))=0$ because $x'<y<f(x')$.
	By Definition \ref{def:dimension2}(1,2), we get $\dim ((x,y))=\max\{\dim((x,x')),\dim(\{x'\}),\dim((x',y))\}=0$.
	On the other hand, $y>f(x)$, which contradicts the definition of $f(x)$.
	We have shown that $f|_{K_i}$ is constant.

	Since $f|_{K_i}$ is constant for every $1 \leq i \leq k$, $Q:=f(J) \subseteq \overline{M}$ is a finite set.
	We want to show $f(x)=\inf\{q \in Q\;|\; x<q\}$ for every $x \in J$.
	
	Assume for contradiction that $f(x) \neq \inf\{q \in Q\;|\; x<q\}$.
	There exists $q \in Q$ such that $x<q<f(x)$.
	Since $Q=f(J)$, there exists $x' \in J$ such that $q=f(x')$.
	If $x'>x$, we have $x<x'<f(x')=q$.
	Choose $y \in J$ with $q<y<f(x)$.
	We have $\dim ((x,y))=0$ by the definition of $f(x)$.
	Then we get $\dim((x',y))=0$ by Definition \ref{def:dimension}(2) because $(x',y) \subseteq (x,y)$, which contradicts the inequality $f(x')<y$ and the definition of $f$.
	Next suppose $x'<x$.
	Choose $y \in J$ such that $q<y<f(x)$.
	We have $\dim ((x',x))=0$ because $x<f(x')$.
	We have $\dim((x,y))=0$ because $y<f(x)$.
	By Definition \ref{def:dimension}(1,2), we have $\dim((x',y))=0$, which contradicts the inequality $f(x')<y$.
	We have shown the equality $f(x)=\inf\{q \in Q\;|\; x<q\}$.

	We want to show $Q=\{b\}$.
	First assume for contradiction that $Q$ has at least two points.
	Let $c_1$ be the leftmost point of $Q$ and $c_2$ be the second leftmost point of $Q$.
	We have $f(x)=c_1$ for every $x \in J$ with $a<x<c_1$ and $f(x)=c_2$ for every $x \in J$ with $c_1 \leq x < c_2$.
	By Lemma \ref{lem:weak0}, $\dim (\{x \in J\;|\;a<x<c_1\})=0$ and $\dim (\{x \in J\;|\;c_1<x<c_2\})=0$.
	By Definition \ref{def:dimension}(1,2), we get $\dim(\{x \in J\;|\; a<x<c_2\})=0$ in both the cases in which $c_1 \in M$ and $c_1 \in \overline{M} \setminus M$.
	Choose $d_1,d_2 \in J$ so that $d_1<c_1<d_2<c_2$.
	We have $\dim((d_1,d_2))=0$, which contradicts $f(d_1)=c_1<d_2$.
	We have shown that $Q$ is a singleton.
	
	If the unique point $q_0$ is smaller than $b$.
	We have $x>f(x)$ for every $x \in J$ with $q_0<x<b$.
	This inequality contradicts the inequality $x<f(x)$, which has been already proved.
	We have shown that $Q=\{b\}$.
	This implies $\dim((t,b))=0$ for every $t \in M$ with $a<t<b$.
	By \cite[Lemma 3.2]{Fuji_dim}, we have $\dim J=\dim ((a,b))=0$.
\end{proof}

Let $\mathcal M=(M,<,\ldots)$ be a weakly o-minimal structure admitting strong cell decomposition.
A definable subset $I$ of $M$ is called \textit{self-sufficient} if, it is infinite and, for every infinite definable sets $J_1 \subseteq I$ and $J_2 \subseteq M \setminus I$, there does not exist a definable bijection between $J_1$ and $J_2$. 
Observe that, if $M \setminus I$ has an empty interior, $I$ is self-sufficient.

The following two lemmas are generalizations of \cite[Lemma 3.15, Lemma 3.16]{Fuji_dim}.
The proofs are almost same as those of the original lemmas.
We use the strong monotonicity \cite[Lemma 2.6(b)]{Wencel} instead of the monotonicity theorem for o-minimal structures.
Therefore, we omit the proofs of these lemmas.

\begin{lemma}\label{lem:self_weak}
	Let $\mathcal M=(M,<,\ldots)$ be a weakly o-minimal structure admitting strong cell decomposition property.
	An infinite definable subset $I$ of $M$ is self-sufficient if and only if, for every definable subset $J$ of $I$ and every definable function $f:J \to M \setminus I$, the image $f(J)$ is a finite set.
\end{lemma}

\begin{lemma}\label{lem:omin_equiv_weak}
	Let $\mathcal M=(M,<,\ldots)$ be a weakly o-minimal structure admitting strong cell decomposition.
	Let $I \subseteq M$ be a nonempty definable set having nonempty interior.
	The function $\mydim[I]$ is a dimension function if and only if $I$ is self-sufficient.
\end{lemma}

Let $\mathcal M=(M,<,\ldots)$ be a weakly o-minimal structure admitting strong cell decomposition property.
We define the equivalence relation $\myfineq^{\mathcal M}$ on $\myDef_1(\mathcal M)$ by $I_1 \myfineq^{\mathcal M} I_2$ if and only if $(I_1 \setminus I_2) \cup (I_2 \setminus I_1)$ is a finite set for  $I_1,I_2 \in \myDef_1(\mathcal M)$.
We drop the superscript $\mathcal M$ of $\myfineq^{\mathcal M}$ if it is clear from the context.
The following lemmas are shown in the same manner as \cite[Lemma 3.17, Lemma 3.18]{Fuji_dim}.

\begin{lemma}\label{lem:fineq_weak}
	Let $I_1,I_2 \in \myDef_1(\mathcal M)$ with $I_1 \myfineq I_2$.
	$I_1$ is self-sufficient if and only if $I_2$ is so.
\end{lemma}

\begin{lemma}\label{lem:well_defined_weak}
	Let $\mathcal M$ be a weakly o-minimal structure admitting strong cell decomposition.
	For self-sufficient sets $I_1$ and $I_2$, the equality $\mydim[I_1]=\mydim[I_2]$ holds if and only if $I_1 \myfineq I_2$.
\end{lemma}

Let $\mySelf(\mathcal M)$ be the set of the equivalence class of self-sufficient sets under the equivalence relation $\myfineq$.
The following corollary is easily shown in the same manner as \cite[Proposition 3.19]{Fuji_dim} using Lemma \ref{lem:well_defined_weak}, Lemma \ref{lem:omin_dimform_weak} and Lemma \ref{lem:omin_equiv_weak}. 

\begin{proposition}\label{prop:weakly_omin_equiv}
	Let $\mathcal M$ be a weakly o-minimal structure admitting strong cell decomposition property.
	Then, there exists a one-to-one correspondence of $\mySet(\mathcal M)$ with $\mySelf(\mathcal M)$. 
\end{proposition}

\section{$\myNum(\mathcal M)$ and $\myNum(\overline{\mathcal M})$}\label{sec:ori}
Let us investigate the relation between $\myNum(\mathcal M)$ and $\myNum(\overline{\mathcal M})$.
\begin{lemma}\label{lem:self_sufficient_gen}
	Let $\mathcal M=(M,<,\ldots)$ be a weakly o-minimal structure admitting strong cell decomposition.
	Let $\overline{\mathcal M}=(\overline{M},<,\ldots)$ be the canonical o-minimal extension of $\mathcal M$.
	Let $I$ be a self-sufficient $\overline{\mathcal M}$-definable subset of $\overline{M}$.
	Then, $I \cap M$ is a self-sufficient $\mathcal M$-definable subset of $M$.
\end{lemma}
\begin{proof}
	Put $J:=I \cap M$.
	It is easy to see that $J$ is $\mathcal M$-definable.
	
	Let $J_1$ be a definable subset of $J$ and $f:J \to M \setminus I$ be a definable function.
	We only have to prove that $f(J_1)$ is a finite set by Lemma \ref{lem:self_weak}.
	By \cite[Lemma 2.6(b)]{Wencel}, we can decompose $J_1$ into a finite set $F$ and finitely many definable open convex sets $X_1,\ldots, X_k$ such that the restriction $f|_{X_i}$ of $f$ to $X_i$ is strongly continuous and monotone.
	Let $\overline{X}_i$ be the completion of $X_i$ defined in \cite[p.143]{Wencel}.
	By the definition of strong continuity, for every $1 \leq i \leq k$, there exists an $\overline{\mathcal M}$-definable continuous extension $\overline{f}_i: \overline{X}_i \to \overline{M}$ of $f|_{X_i}$.
	
	We show that $T:=I \cap \bigcup_{i=1}^k \overline{f}_i(\overline{X}_i)$ is a finite set.
	Assume for contradiction that $T$ is infinite.
	There exists $1 \leq i \leq k$ such that $S:=I \cap \overline{f}_i(\overline{X}_i)$ is infinite.
	The inverse image $(\overline{f}_i)^{-1}(S)$ is infinite.
	Since $(\overline{f}_i)^{-1}(S)$ is definable in the o-minimal structure $\overline{\mathcal M}$, it contains a nonempty open interval.
	Since $M$ is dense in $\overline{M}$, there exists $x_0 \in M \cap (\overline{f}_i)^{-1}(S) \subseteq X_i$.
	This means that $f(x_0) =\overline{f}_i(x_0) \in I$, which contradicts the setting that the target space of $f$ is $M \setminus I$. 
	We have shown that $T$ is a finite set.
	
	Since $I$ is self-sufficient and $\overline{f}_i(\overline{X_i} \setminus T) \subseteq \overline{M} \setminus I$, the image $\overline{f}_i(\overline{X_i} \setminus T)$ is a finite set for each $1 \leq i \leq k$ by Lemma \ref{lem:self_weak}.
	Since $$f(J_1) \subseteq f(F) \cup \bigcup_{i=1}^k (\overline{f}_i(\overline{X_i} \setminus T) \cup \overline{f}_i(\overline{X_i} \cap T)),$$
	$f(J_1)$ is a finite set.
\end{proof}

\begin{proposition}\label{thm:correspond}
	Let $\mathcal M=(M,<,\ldots)$ be a weakly o-minimal structure admitting strong cell decomposition.
	Let $\overline{\mathcal M}=(\overline{M},<,\ldots)$ be the canonical o-minimal extension of $\mathcal M$.
	\begin{enumerate}
		\item[(i)] $\myNum(\mathcal M) \geq \myNum(\overline{\mathcal M})$
		\item[(ii)] For every $\dim_{\overline{\mathcal M}} \in \mySet(\overline{\mathcal M})$, let us consider the function $\dim_{\mathcal M}:\myDef(\mathcal M) \to \mathbb N \cup \{-\infty\}$ defined by $$\dim_{\mathcal M}(X) = \min\{\dim_{\overline{\mathcal M}} (\overline{X})\;|\; \ \overline{X} \in \myDef_n(\overline{\mathcal M}) \text{ with } \overline{X} \cap M^n = X\}$$
		for $X \in \myDef_n(\mathcal M)$.
		Then, $\dim_{\mathcal M}$ belongs to $\mySet(\mathcal M)$.
	\end{enumerate}
\end{proposition}
\begin{proof}
	(i)
	By Proposition \ref{prop:weakly_omin_equiv}, it suffices to construct an injection $\eta:\mySelf(\overline{\mathcal M}) \to \mySelf({\mathcal M})$.
	
Let $I_1$ and $I_2$ be two self-sufficient $\overline{\mathcal M}$-definable subsets of $\overline{M}$.
Recall that $I_i \cap M$ is $\mathcal M$-definable for $i=1,2$.
It is easy to see that $I_1 \myfineq^{\overline{\mathcal M}} I_2$ if and only if $I_1 \cap M \myfineq^{\mathcal M} I_2 \cap M$ using the fact that $M$ is dense in $\overline{M}$. 
We put $\eta([I])=[I \cap M]$, where $I$ is a self-sufficient $\overline{\mathcal M}$-definable subset of $\overline{M}$.
The map $\eta:\mySelf(\overline{\mathcal M}) \to \mySelf({\mathcal M})$ is a well-defined injective map by the above observation and Lemma \ref{lem:self_sufficient_gen}.

	(ii)
By Proposition \ref{prop:weakly_omin_equiv}, there exists a self-sufficient $\overline{\mathcal M}$-definable subset $I$ of $\overline{M}$ such that $\mydim[I]=\dim_{\overline{\mathcal M}}$ by Lemma \ref{lem:omin_dimform_weak} and Lemma \ref{lem:omin_equiv_weak}.
%

By Lemma \ref{lem:self_sufficient_gen}, $J:=I \cap M$ is $\mathcal M$-definable and self-sufficient.
Put $\dim_{\mathcal M}=\mydim[J]$, which is a dimension function by Lemma \ref{lem:omin_equiv_weak}.
We prove the equality
\begin{equation}
	\dim_{\mathcal M}(X) = \min\{\dim_{\overline{\mathcal M}} (\overline{X})\;|\; \ \overline{X} \in \myDef_n(\overline{\mathcal M}) \text{ with } \overline{X} \cap M^n = X\} \label{eq:acd}
\end{equation}
holds for $X \in \myDef_n(\mathcal M)$.
Since $\dim_{\overline{\mathcal M}}$ and $\dim_{{\mathcal M}}$ are dimension functions, we can use the relations in Definition \ref{def:dimension}.
We use this fact without notice in this proof.

We denote the right hand of equality (\ref{eq:acd}) by $d(X)$.
The following claim is obvious from the definition of $d(X)$.
\medskip

\textbf{Claim 1.}
Let $\overline{X}, \overline{X'} \in \myDef_n(\overline{\mathcal M})$.
If $\overline{X'} \subseteq \overline{X}$, $\overline{X} \cap M^n = \overline{X'} \cap M^n =X$ and $\dim_{\overline{\mathcal M}} (\overline{X})=d(X)$, then $\dim_{\overline{\mathcal M}} (\overline{X'})=d(X)$.
\medskip

Put $I_0:=I$, $I_1 := \overline{M} \setminus I$, $J_0:=J$ and $J_1:=M \setminus J$. 
For every positive integer $n$, $0<k \leq n$ and $\tau \in \{0,1\}^n$, the $k$-th coordinate of $\tau$ is denoted by $\tau(k)$ and put $I \langle \tau \rangle:=\prod_{i=1}^{n} I_{\tau(i)}$.
We define $J \langle \tau \rangle$ in the same manner.
The following claim is easily proved by induction on $n$.
See \cite[p.455]{Wencel2} for the definition of basic cells.
\medskip

\textbf{Claim 2.} Let $\tau \in \{0,1\}^n$.
Let $C$ be a basic $\langle i_1,\ldots, i_n\rangle$-cell in $\overline{M}^n$ contained in $I\langle \tau \rangle$ such that $C \cap M^n \neq \emptyset$.
Then $C \cap M^n$ is a strong $\langle i_1,\ldots, i_n\rangle$-cell in $M^n$ contained in $J\langle \tau \rangle$ such that $\dim_{\mathcal M}(C \cap M^n)=\dim_{\overline{\mathcal M}}(C)=\sum_{1 \leq j \leq n, \tau(j)=0}i_j$.
\medskip

Let $\overline{X} \in \myDef_n(\overline{\mathcal M})$ such that $\overline{X} \cap M^n = X$ and $\dim_{\overline{\mathcal M}} (\overline{X})=d(X)$.
By \cite[Proposition 2.3]{Wencel2}, there exists a cell decomposition $\mathcal C$ into basic cells partitioning $\{\overline{X}\} \cup \{I\langle \tau \rangle \;|\; \tau \in \{0,1\}^n\}$.

Put $\mathcal D=\{C \in \mathcal C\;|\; C \subseteq \overline{X}\}$ and $\mathcal E=\{C \in \mathcal D\;|\; C \cap M^n \neq \emptyset\}$.
Observe that $\overline{X}=\bigcup_{C \in \mathcal D}C$.
Put $\overline{X'}:=\bigcup_{C \in \mathcal E}C$, then $\overline{X'} \subseteq \overline{X}$ and $\overline{X'} \cap M^n =X$.
We have $\dim_{\overline{\mathcal M}} (\overline{X'})=d(X)$ by Claim 1.
Using Claim 2, we get
\begin{align*}
	d(X) &= \dim_{\overline{\mathcal M}} (\overline{X'}) = \max\{\dim_{\overline{\mathcal M}}(C)\;|\; C \in \mathcal E\}\\
	&= \max\{\dim_{{\mathcal M}}(C \cap M^n)\;|\; C \in \mathcal E\} = \dim_{\mathcal M}(\bigcup_{C \in \mathcal E}(C \cap M^n))\\
	&=\dim_{\mathcal M}(X).
\end{align*}
We have shown equality (\ref{eq:acd}).
\end{proof}

The following proposition shows that the inequality in Proposition \ref{prop:weakly_omin_equiv}(1) is strict in general.
 
\begin{proposition}\label{prop:path}
	Let $\mathcal L=\{<,U\}$ be a language in which the symbols are binary predicates.
	We interpret these symbols on $\mathbb Q$ as follows:
	The symbol $<$ defines the standard order on $\mathbb Q$.
	For $x,y \in \mathbb Q$, the relation $U(x,y)$ holds if and only if $y<x+\pi$.
	Put $\mathcal M:=(\mathbb Q,<,U)$, which is an $\mathcal L$-structure.
	Then, $\mathcal M$ is a weakly o-minimal structure admitting strong cell decomposition, $\myNum(\mathcal M)=\aleph_0$ and $\myNum(\overline{\mathcal M})=1$.
\end{proposition}
\begin{proof}
	Consider the following property for $n>0$ and a definable subset $X$ of $\mathbb Q^n$:
	\begin{description}
		\item[$(*)_{n,X}$] For every definable subset $X$ of $\mathbb Q^n$, there exist finitely many strong cells $C_1, \ldots, C_k$ such that $X=\bigcup_{i=1}^k C_i$, $C_i \cap C_j = \emptyset$ for $i \neq j$, and either $\Pi_1^n(C_i)=\Pi_1^n(C_j)$ or $\Pi_1^n(C_i) \cap \Pi_1^n(C_j) = \emptyset$ for $i \neq j$. 
		\item[$(**)_n$] Let $\mathcal A=\{A_i\}_{i=1}^m$ be a finite family of definable subsets of $\mathbb Q^n$. 
		There exists a strong cell decomposition partitioning $\mathcal A$.
	\end{description}
	
	Observe that $\mathcal L(M)$-terms are either free variables or constants.
	First we show the following claim:
	\medskip
	
	\textbf{Claim 1.} Every definable subset $X$ of $\mathbb Q^n$ is of the form:
	\begin{equation}
		X=\bigcup_{i=1}^l \bigcap_{j=1}^{k_i}\{\overline{x} \in \mathbb Q^n\;|\; *_{ij}(t_{ij1}(\overline{x}),t_{ij2}(\overline{x}))\}, \label{eq:mm}
	\end{equation}
	where $*_{ij} \in \{<,=,U,\neg U\}$ and $t_{ijk}(\overline{x})$ are $\mathcal L(M)$-terms for $k=1,2$.
	
	Furthermore, if $n>1$, $l=1$ and $\mathcal M$ possesses property $(**)_{n-1}$, then $\mathcal M$ enjoys the property $(*)_{n,X}$.  
	\begin{proof}[Proof of Claim 1]
		First we prove the first statement of Claim 1.
		Let $\mathcal S_n$ be the collection of definable subsets of $\mathbb Q^n$ of the form in Claim 1.
		Put $\mathcal S:=\bigcup_{n=1}^{\infty}\mathcal S_n$.
		We only have to show that $\mathcal S$ is closed under taking finite union, negation, Cartesian product and projection image.
		It is obvious that $\mathcal S$ is closed under taking finite union, negation and Cartesian product.
		We only have to show that, if $X$ is of the form in Claim 1, then $\Pi_1^n(X)$ is of the form in Claim 1.
		
		We can easily reduce to the case where $l=1$.
		We may assume that $X$ is of the form:
		\begin{equation}
		\bigcap_{j=1}^{k}\{\overline{x} \in \mathbb Q^n\;|\; *_{j}(t_{j1}(\overline{x}),t_{j2}(\overline{x}))\}, \label{eq:1}
		\end{equation}
		where $*_{j} \in \{<,=,U,\neg U\}$ and $t_{jp}(\overline{x})$ are $\mathcal L(M)$-terms for $p=1,2$.
		Let $x_i$ be the $i$-th coordinate of $\overline{x}$ for $1 \leq i \leq n$.
		Put $\overline{x'}:=(x_1,\ldots,x_{n-1})$.
		
		Put $\phi_1(x,y):=\text{`}x=y\text{'}$, $\phi_2(x,y):=\text{`}x<y\text{'}$ and $\phi_3(x,y):=\text{`}U(x,y)\text{'}$.
		For $1 \leq j \leq 3$, we put $\phi_{j,1}(x,y)=\phi_j(x,y)$ and $\phi_{j,0}(x,y)=\neg\phi_j(x,y)$.
		Let $\mathcal T_X$ be the collection of the terms appearing in the definition of $X$ other than $x_n$.
		By the definition of $\mathcal T_x$, the variables appearing the terms in $\mathcal T_X$ belong to $\{x_1,\ldots, x_{n-1}\}$.
		For every $\tau \in \mathcal P_X:=(\{1,2,3\} \times \mathcal T_X \times \mathcal T_X)^{\{0,1\}}$, 
		Put $$\psi \langle \tau \rangle(\overline{x'}) :=\bigwedge_{i=1}^3\bigwedge_{(t_1(\overline{x'}),t_2(\overline{x'})) \in \mathcal T_X \times \mathcal T_X} \phi_{i,\tau(i,t_1(\overline{x'}),t_2(\overline{x'}))}(t_1(\overline{x'}),t_2(\overline{x'})).$$
		Put $Z\langle \tau \rangle:= \{\overline{z} \in \mathbb Q^{n-1}\;|\; \mathcal M \models \psi \langle \tau \rangle(\overline{z})\}$.
		Observe that the order of the entries in $\{t(\overline{z})\;|\; t(\overline{x'}) \in \mathcal T_X\} \cup \{t(\overline{z})+\pi\;|\; t(\overline{x'}) \in \mathcal T_X\}$ is independent of the choice of $\overline{z} \in Z\langle \tau \rangle$.
		Therefore, if $\mathcal M \models \exists y\ \bigwedge_{j=1}^k *_j(t_{j1}(\overline{z}_0,y), t_{j2}(\overline{z}_0,y))$ holds for some $\overline{z}_0 \in Z\langle \tau \rangle$, $\mathcal M \models \exists y\ \bigwedge_{j=1}^k *_j(t_{j1}(\overline{z},y), t_{j2}(\overline{z},y))$ holds for every $\overline{z} \in Z\langle \tau \rangle$.
		This implies that $\Pi_1^n(X)$ is a union of definable sets of the form $Z\langle \tau \rangle$.
		We have proved Claim 1 other than the `furthermore' part.
		
		Let us move on to the proof of  the `furthermore' part.
		It is easy to show that, if $C \subseteq \mathbb Q^{n-1}$ be a strong cell such that $C \subseteq Z \langle \tau \rangle$ for some $\tau \in \mathcal P_X$, then $X \cap (C \times \mathbb Q)$ is a strong cell.
		We omit the proof.
		By the assumption, there exists a strong cell decomposition $\mathcal C$ partitioning the family $\{Z\langle \tau \rangle\;|\; \tau \in \mathcal P_X\}$.
		The family of cells given by $\{X \cap (C \times \mathbb Q)\;|\; C \subseteq \Pi_1^n(X)\}$ is the family of strong cells satisfying the conditions in property $(*)_{n,X}$. 
	\end{proof}
	
	Claim 1 immediately implies that $\mathcal M$ is weakly o-minimal.
	\medskip
	
	\textbf{Claim 2.} Let $X_i$ be definable subsets of $\mathbb Q^n$ for $1 \leq i \leq k$.
	Put $X=\bigcup_{i=1}^k X_i$
	Suppose that $\mathcal M$ enjoys property $(**)_{n-1}$ and property $(*)_{n,X_i}$ for $1 \leq i \leq k$. 
	Then, $\mathcal M$ possesses $(*)_{n,X}$.
	\begin{proof}[Proof of Claim 2]
		Let $\mathcal C_i$ be the strong cells decomposition partitioning $X_i$ satisfying the conditions in $(*)_{n,X_i}$
		Apply $(**)_{n-1}$ to the family $\{\pi(C)\;|\; 1 \leq i \leq k, C \in \mathcal C_i\}$.
		We get a decomposition $\mathcal D$ of $\mathbb Q^{n-1}$ partitioning the above family.
		By replacing $\mathcal C_i$ with $\{C \cap (D \times \mathbb Q)\;|\; D \in \mathcal D, C \in \mathcal C_i, D \subseteq \Pi_1^n(C)\}$, we may assume that, for every $1 \leq i \leq j \leq k$, $C \in \mathcal C_i$ and $C' \in \mathcal C_j$, we may assume that either $\Pi_1^n(C)=\Pi_1^n(C')$ or $\Pi_1^n(C) \cap \Pi_1^n(C')=\emptyset$ holds.
		
		Put $$\mathcal E:=\{\Pi_1^n(C)\;|\; C \in \bigcup_{i=1}^k \mathcal C_i\}.$$
		We use this notation $\mathcal E$ by replacing $\mathcal C_i$ with a new family.
		Recall that the strong cells $C$ are either the graph of a definable strongly continuous function on $\pi(C)$ or the region sandwiched by the graph of definable strongly continuous functions on $\pi(C)$.
		For every $D \in \mathcal E$, let $\mathcal F_D$ be the collections of definable strongly continuous functions defined on $D$ appearing in the definition of strong cells $C$ in $\bigcup_{i=1}^k \mathcal C_i$ such that $\pi(C)=D$.
		For every $D \in \mathcal E$ and $f,g:D \to \overline{M}\cup \{\pm \infty\} \in \mathcal F_D$, put $Y \langle D,f,g,0 \rangle:=\{x \in D\;|\; f(x)=f(g)\}$ and $Y \langle D,f,g,1 \rangle:=\{x \in D\;|\; f(x)<f(g)\}$.
		Apply $(**)_{n-1}$ to the family $\{Y \langle D,f,g,1 \rangle\;|\; D \in \mathcal E, f,g \in \mathcal F_D\}$.
		We get a decomposition $\mathcal D'$ of $\mathbb Q^{n-1}$ partitioning the above family.
		In the same manner as the (omitted) proof of \cite[Fact 2.5]{Wencel}, we may further assume that one of the following conditions are satisfied for $D \in \mathcal D'$ and $f \in \mathcal F_D$:
		\begin{center}
			$\forall x \in D\ f(x) \in M$ and $\forall x \in D \ f(x) \in (\overline{M} \setminus M \cup \{\pm \infty\})$.
		\end{center}
		We say that $f$ is \textit{$M$-valued} if the first condition of the above two holds.
		By considering $\{C \cap (D \times Q)\;|\; C \in \mathcal C_i, D \in \mathcal D' \text{ with }D \subseteq \pi(C)\}$ instead of $\mathcal C_i$, we may assume that, for each $D \in \mathcal E$, the functions in $\mathcal F_D$ can be enumerated as $f_1, \ldots, f_{l_D}$ so that $f_1(x)<f_2(x)<\cdots<f_{l_D}(x)$ holds for every $x \in D$.
		Let us consider the family $$\bigcup_{D \in \mathcal D'} \{\Gamma(f)\;|\; f \in \mathcal F_D \text{ and }f \text{ is }M \text{-valued} \} \cup \{(f_j,f_{j+1})_D\;|\; 1 \leq j < l_D\},$$
		where $\Gamma(f)$ is the graph of $f$ and $(f,g)_D:=\{(x,y) \in D \times \mathbb Q\;|\; f(x)<y<g(x)\}$.
		This family is a strong cell decomposition partitioning $X$.
	\end{proof}

	Next we want to show that $\mathcal M$ admits strong cell decomposition.
	The case $n=1$ is easy because $\mathcal M$ is weakly o-minimal.
	Suppose $n>1$.
	By Claim 1, Claim 2 and the induction hypothesis, $\mathcal M$ possesses property $(*)_{n,X}$ for every definable subset $X$ of $\mathbb Q^n$.
	
	We prove $\mathcal M$ enjoys property $(**)_n$.
	Put $A_i^{(0)}=A_i$ and $A_i^{(1)}=\mathbb Q^n \setminus A_i$.
	For every $\iota \in \{0,1\}^m$, $\iota(i)$ is the $i$-th element of the sequence $\iota$.
	For every $\iota \in \{0,1\}^m$, put $X[\iota]=\bigcap_{i=1}^m A_i^{(\iota(i))}$ and apply $(*)_{n,X[\iota]}$ to this set.
	There exists a decomposition $\mathcal C[\iota]$ of $X[\iota]$ satisfying the conditions in $(*)_{n,X[\iota]}$.
	Apply $(**)_{n-1}$ to the family $\{\pi(C)\;|\; \iota \in \{0,1\}^m, C \in \mathcal C[\iota]\}$.
	We get a strong cell decomposition $\mathcal D$ of $\mathbb Q^{n-1}$ partitioning this family.
	It is obvious that $$\{C \cap (D \times \mathbb Q)\;|\; \iota \in \{0,1\}^m,\ C \in \mathcal C[\iota], D \in \mathcal D \text{ with } D \subseteq \pi(C)\}$$
	is a strong cell decomposition partitioning $\mathcal A$.
	
	We next show that $\myNum(\mathcal M)=\aleph_0$.
	We give a complete proof though the proof for this fact is very similar to the proof of \cite[Proposition 3.22]{Fuji_dim}.
	
	We prove that $I_a:=\{x \in \mathbb Q\;|\;x>a\}$ is self-sufficient.
	Let $J \subset I_a$ be an infinite definable set and $f:J \to \mathbb Q \setminus I_a$ is a definable function.
	We want to show that $f(J)$ is a finite set.
	Let $X$ be the graph of $f$.
	By Claim 1, $X$ is of the form given in equality (\ref{eq:mm}).
	There exist definable functions $X_1,\ldots X_l$ of the form (\ref{eq:1}) such that $X=\bigcup_{i=1}^l X_i$.
	$X_i$ is the graph of restriction of $f$ to $\Pi_1^2(X_i)$.
	If $f(\Pi_1^2(X_i))$ is a finite set for each $1 \leq i \leq l$, $f(J)$ is also a finite set.
	Therefore, we may assume that $X$ is of the form (\ref{eq:1}).
	Since $X \subseteq I_a \times (\mathbb Q \setminus I_a)$, the inequality $x>y$ holds for every $(x,y) \in X$.
	Since $X_a:=\{y \in \mathbb Q\;|\;(x,y) \in X\}$ is a singleton, one of $*_j(t_{j1}(\overline{x}),t_{j2}(\overline{x}))$ should be of the form $\text{`}y=b\text{'}$ for some $b \in M$.
	This implies that $f$ is a constant function.
	We have shown that $f(J)$ is a finite set, which implies $I_a$ is self-sufficient.
	Obviously, we have $I_a \not\myfineq^{\mathcal M} I_b$ if $a \neq b$.
	Therefore, we have $\myNum(\mathcal M) \geq |\mathbb Q|=\aleph_0$ by Proposition \ref{prop:weakly_omin_equiv}.
	
	We show the inequality $\myNum(\mathcal M) \leq \aleph_0$.
	Let $\dim$ be a dimension function on $\mathcal M$.
	We define $\eta(\dim) \in \myDef_2(\mathcal M)$ as follows:
	\begin{align*}
		&\eta(\dim):=\{(a,b) \in M^2\;|\; \dim ((a,b))=0\}.
	\end{align*}
	Observe that $\eta(\dim)$ is a definable set by Definition \ref{def:dimension}(4).
	Lemma \ref{lem:weak_omin_basic} implies that $\eta$ is injective.
	$\myDef_2(\mathcal M)$ is of the cardinality smaller than or equal to the cardinality of the set of finite sequences of elements from the union of $M$ with the finite set consisting of logical symbols, two variables and $<$. 
	Therefore, we have $|\myDef_2(\mathcal M)|=|\mathbb Q|$ (under the assumption of the axiom of choice).
	We have $\myNum(\mathcal M) \leq \aleph_0$.
	
	The next task is to show that the canonical o-minimal extension $\overline{\mathcal M}=(M,<,\ldots)$ of $\mathcal M$ has only one self-sufficient set up to the equivalence condition $\myfineq^{\overline{\mathcal M}}$.
	Observe that $\sup\{y \in \mathbb Q\;|\; U(x,y)\}=x+\pi$ for every $x \in \mathbb Q$.
	Therefore, the map $\overline{M} \ni x \mapsto x+\pi \in  \overline{M}$ is definable in $\overline{\mathcal M}$.
	It is obvious that every cofinite $\overline{M}$-definable subset $I$ of $\overline{M}$ is self-sufficient and equivalent to $\overline{M}$ under $\myfineq^{\overline{\mathcal M}}$.
	
	Let $I$ be an infinite $\overline{M}$-definable subset $I$ of $\overline{M}$ whose complement is infinite.
	The remaining task is to show that $I$ is not self-sufficient.
	By o-minimality of $\overline{\mathcal M}$, one of $I$ and $\overline{M} \setminus I$ contains an open interval of the form $(-\infty,b)$ for some $b \in \overline{M}$.
	Suppose $I$ contains such an open interval.
	We can prove that $I$ is not self-sufficient in a similar manner in the other case.
	Since $\overline{M} \setminus I$ is an infinite $\overline{\mathcal M}$-definable set, there exist a finite set $F$ and open intervals $J_1, \ldots, J_m$ with $m \geq 1$ such that $\overline{M} \setminus I$ is a disjoint union of $F$ and $J_1, \ldots, J_m$.
	Let $J_1$ be the leftmost open interval among $J_1, \ldots, J_m$.
	Let $c$ be the left endpoint of $J_1$.
	Observe that $(-\infty,c) \setminus I$ is a finite set.
	Ut is easy to show that $I$ is self-sufficient if and only if $I\cup ((-\infty,c) \setminus I)$ is self-sufficient.
	We omit the proof.
	We may assume that $(-\infty,c) \subseteq I$ by considering $I\cup ((-\infty,c) \setminus I)$ instead of $I$.
	Choose $d \in \overline{M}$ so that $c<d<c+\pi$ and $(c,d) \subseteq J_1$.
	Put $K_1=(c-\pi,d-\pi)$ and $K_2=(c,d)$, then we have $K_1 \subseteq I$ and $K_2 \cap I=\emptyset$ and the map given by $x \mapsto x+\pi$ is a definable bijection between $K_1$ and $K_2$.
	This witnesses that $I$ is not self-sufficient.
	We have shown that $\mySelf(\overline{\mathcal M})$ is a singleton.
	
	We get $\myNum(\overline{\mathcal M})=1$ by Proposition \ref{prop:weakly_omin_equiv}.
\end{proof}

\end{document}